\documentclass[a4paper,11pt]{amsart}
\usepackage[utf8]{inputenc}
\usepackage[T1]{fontenc}
\usepackage[english]{babel}
\usepackage{textcomp}
\usepackage{lmodern}
\usepackage{amsthm}
\usepackage{amsmath}
\usepackage{amssymb}
\usepackage{mathrsfs}
\usepackage{enumerate}
\usepackage{dsfont}
\usepackage[titletoc]{appendix}
\usepackage[left= 2.25cm,right=2.25cm,top=2.75cm,bottom=2.75cm]{geometry}

\newtheorem{theorem}{Theorem}[section]
\newtheorem{lemma}[theorem]{Lemma}

\newtheorem{proposition}[theorem]{Proposition}

\def\eps{\varepsilon}
\newcommand\ds{ \displaystyle }

\newcommand\NN{\mathbb{N}}
\newcommand\RR{\mathbb{R}}
\newcommand\TT{\mathbb{T}}

\newcommand\bB{{\mathbf B}}

\newcommand\bE{{\mathbf E}}

\newcommand\bG{{\mathbf G}}

\newcommand\bU{{\mathbf U}}

\newcommand\bu{{\mathbf u}}

\newcommand\bw{{\mathbf w}}

\newcommand\ZZ{{{\rm Z}\kern-.28em{\rm Z}}}
\newcommand\ee{ \mathbf{e} }
\newcommand\xx{ \mathbf{x} }

\newcommand\vv{ \mathbf{v} }

\title[The Vlasov-Poisson system with strong magnetic field]{On the
  asymptotic limit of the three dimensional Vlasov-Poisson system for
  large magnetic field : formal derivation.}
\author{Pierre DEGOND and Francis FILBET}
\date{\today}

\def\signFF{\bigskip\bigskip\hspace{80mm}
\vbox{{\sc Francis Filbet\par\vspace{3mm}
Institut de Math\'ematiques de Toulouse,\par 
Universit\'e Toulouse III, \par Institut Universitaire de France, \par
118, route de Narbonne\par
F-31062 Toulouse cedex,  FRANCE
\par\vspace{3mm}e-mail:} francis.filbet@math.univ-toulouse.fr }}

\def\signPD{\bigskip\bigskip\hspace{80mm}
\vbox{{\sc Pierre Degond \par\vspace{3mm}
Department of Mathematics,\par
Imperial College of London,\par
London SW7 2AZ,\par
UNITED KINGDOM
\par\vspace{3mm}e-mail:}  pdegond@imperial.ac.uk }}

\begin{document}

\begin{abstract}
This paper establishes the long time asymptotic limit of the three
dimensional Vlasov–Poisson equation with strong external magnetic
field. The guiding center approximation is investigated in the three
dimensional case with a non-constant magnetic field.  In the long time
asymptotic limit, the motion can
be split in two parts : one stationary flow along the lines of the
magnetic field and the guiding center motion in the orthogonal plane
of the magnetic field where classical drift velocities are recovered. We discuss in particular the effect of nonconstant external magnetic fields. 
\end{abstract}

\maketitle

\vspace{0.1cm}



\tableofcontents


\section{Introduction}
\label{sec:1}
We consider a plasma  confined by a strong external nonconstant magnetic
field, hence the charged gas evolves under its self-consistent
electrostatic field and the confining magnetic field.  This configuration is typical of a
tokamak plasma \cite{bellan_2006_fundamentals, miyamoto_2006_plasma}
where the magnetic field is used to confine particles inside the core
of the device.

We assume that on the time scale we consider, collisions can be
neglected  both for ions and electrons, hence collective effects are
dominant and the plasma is entirely modelled with kinetic transport
equations, where the unknown is the number density of particles $f\equiv
f(t,\xx,\vv)$  depending on time $t\geq 0$, position
$\xx\in\Omega\subset \RR^3$ and velocity $\vv\in\RR^3$. 

Such a kinetic model provides an appropriate description of turbulent transport 
in a fairly general context,  but it requires to solve a six dimensional problem which leads to a huge 
computational cost. 

To reduce the cost of numerical simulations, it is classical to derive asymptotic models
with a smaller number of variables than the kinetic description. Large
magnetic fields  usually lead to the so-called drift-kinetic limit
\cite{ant_lane_80,bri_hahm_07,haz_ware_78,haz_mei_03} and we refer to \cite{bostan_08,brenier_00,fre_son_97,fre_son_98,fre_rav_son_01,gol_lsr_99}
for recent  mathematical results on this topic. In this regime, due to the large applied magnetic
field, particles are confined along the magnetic field lines and their period of rotation
around these lines (called the cyclotron period) becomes small.  It corresponds to the finite Larmor radius scaling for the Vlasov–Poisson equation, which was introduced
by Fr\'enod and Sonnendrücker in the mathematical literature \cite{fre_son_97,fre_son_98}. The two-dimensional version of the system (obtained when
one restricts to the perpendicular dynamics) and the large magnetic
field limit  were studied in  \cite{fre_rav_son_01} and more recently in \cite{bostan_08,hauray,Daniel00}.
We also refer to the recent work \cite{hauray2} of Hauray and Nouri, dealing with the well-posedness theory with a diffusive
version of a related two dimensional system. A version of the full
three dimensional system describing ions with massless electrons was studied by the author in \cite{Daniel0,Daniel2}.

Here, we formally derive a new asymptotic model under both assumptions 
of large magnetic fields and large time asymptotic limit for the three
dimensional Vlasov-Poisson system. Analogous problem has already been
carefully studied by F. Golse and  L. Saint-Raymond in two dimension \cite{gol_lsr_99,laure0,gol_lsr_03}.

We consider a plasma model in which we focus on the
dynamics of the fast electrons and the magnetic field is assumed to be
given. 

In the large magnetic field regime, the Lorentz force term in the Vlasov equation
is scaled by a large parameter, $1/\varepsilon$, where $\varepsilon$ stands for
the dimensionless ion cyclotron period, {\it i.e.} the rotation period of the electrons about a 
magnetic field line (or Larmor rotation). The so called drift-kinetic or gyro-kinetic 
regimes are reached when $\varepsilon$ tends to zero (see
\cite{haz_ware_78,Lifshitz}). 

Here we are interested in the long time
behavior of the distribution of electrons since they can be considered
as fast particles compared to the characteristic velocity of electrons. In
this limit, the new distribution function  only depends on space, time
and two components of the velocity, corresponding to the parallel
component along the magnetic field line and the magnitude of the
perpendicular velocity. In other words, the distribution function is
independent of the gyro-phase of the perpendicular velocity in the
plane normal to the magnetic field line. This is a consequence of the
ultra-fast cyclotron rotation about the magnetic field lines. 

It is also convenient to express the distribution in terms of the parallel velocity and the magnetic moment or adiabatic invariant, which is proportional to the perpendicular energy divided by the magnitude of the magnetic field. 

Now, the distribution function in these new variables satisfies a
transport equation  with a constraint. A Lagrange multiplier allows to
express this constraint in the differential system. The constraint
expresses that the distribution function is constant along the
trajectories of the fast parallel motion along the magnetic field
lines. 

The derivation of the model roughly follows the following steps: we
first proceed with formal expansions of the distribution function with
respect to the parameter $\varepsilon$.  Now, carrying the Hilbert
expansion procedure through for the distribution function equation is
best done if we change the random velocity variable  into a coordinate
system consisting of the parallel velocity, the energy, and the angle
of rotation or gyrophase around the magnetic field line. 

Thanks to this coordinate change, we show that the leading order term
of the distribution function does not depend on the gyrophase. Next,
we realize that, at each level of the expansion, we are led to
inverting the gyrophase averaging operator
\cite{haz_ware_78,haz_mei_03}. We show that the inverse operator can
only act on functions satisfying a specific solvability condition,
namely that their gyrophase average is zero. We find the asymptotic
model following the classical Hilbert expansion procedure of kinetic
theory. Providing an explicit expression of the Hilbert expansion
procedure is achieved here assuming that the magnetic field only acts
in the $z$ direction. 

The remainder of the paper is organized as follows. In
Section~\ref{sec:2}, we present the scaling which expresses the
assumptions of strong magnetic field and long time asymptotic regime.
Then, we present and comment the main result of this article, namely the asymptotic model.
In Section~\ref{sec:3}, by using Hilbert
expansions we derive the asymptotic model and provide the main
computational steps which lead to the explicit partial differential
system for the limit distribution function.

\section{Scaling and main results}
\label{sec:2}


\subsection{The Vlasov equation in a strong magnetic field}
\label{Vlasov}

We are interested in the dynamics of a single species negatively
charged fast electrons in the plasma. At this stage of the study, 
the coupling with the ions is discarded and the electric field
is given by the Poisson equation whereas the magnetic field is external.

We investigate the asymptotic limit of the Vlasov equation describing
the long time dynamics of the electrons when
they are submitted to an asymptotically large external magnetic field.

Denoting by $m$ the electron mass and by $q$ the negative charge of the electron,
we start from the Vlasov equation
\begin{equation}\label{V_f}
\frac{\partial f}{\partial t} + \vv\cdot\nabla_\xx f +
\frac{q}{m}(\bE\,+\,\vv\times \bB_{\rm ext})\cdot\nabla_\vv f =0,
\end{equation}
where $f\equiv f(t,\xx,\vv)$ is the distribution function and $\xx\in \Omega \subset \RR^3$,
$\vv\in\RR^3$, and $t\in\RR^+$ are respectively the position, velocity, and
time variables. 

Then, we  prescribe an initial datum 
\begin{eqnarray}
f(0,\xx,\vv) = f_{\rm in}(\xx,\vv), \quad \xx \in  \Omega,   \quad \vv\in\RR^3
\label{VIC}
\end{eqnarray}
where $f_{\rm in}$ is the distribution function of particles initially present inside the domain $\Omega$.  

Next, we introduce a set of characteristic scales from which an appropriate
scaling of equation \eqref{V_f} will be derived.

The characteristic length scale of the problem $\overline x$   is the Debye length 
$$
\lambda_D\,\,=\,\,\left(\frac{k_{\mathcal B}\overline{T}}{4\pi\,\overline{n} \,q^2}\right)^{1/2}
$$ 
where $k_{\mathcal{B}}$ is the Boltzmann
constant, $\overline{T}$ is the temperature scale and $\overline{n}$
is the density scale. Then, the characteristic magnitude of the
electric  field can be expressed from $\overline n$ and $\overline x$
by $\overline{E}=4\pi q^2 \overline{n} \,\overline{x} $ and the characteristic
velocity of electrons $\overline{v}$ is the thermal
velocity of the electrons, $v_{th}=(k_{\mathcal
  B}\overline{T}/m)^{1/2}$. 

Therefore, the plasma frequency of electrons satisfies
$$
\omega_p^{-1} = \frac{\overline x}{\overline v},
$$
which corresponds to one time scale.  Moreover,  we  denote by
$\overline{B}$ the characteristic magnitude of the applied magnetic
field and define $\omega_c=\frac{q\overline{B}}{m}$ the characteristic electron
cyclotron frequency, and $\omega_c^{-1}$ corresponds to a second time scale.

Hence we define the new variables and given fields by 
$$
\xx'=\frac{\xx}{\overline x},\quad\vv'=\frac{\vv}{\overline{v}},\quad t'=\frac{t}{\overline{t}},\quad \bE'(t',\xx')=
\frac{\bE(t,\xx)}{\overline{E}},\quad\bB'_{\rm ext}(t',\xx')=\frac{\bB_{\rm ext}(t,\xx)}{\overline{B}}.
$$
Subsequently, letting $\overline{f}=\overline{n}/\overline{v}^3$ the distribution function scale, we
introduce the new unknown $$
f'(t',\xx',\vv')=\frac{f(t,\xx,\vv)}{\overline{f}}.
$$
Inserting all these changes into \eqref{V_f}, dividing by $\omega_p$ and dropping the primes
for clarity, we obtain the dimensionless equation
\begin{equation}
\label{primes}
\frac{1}{\omega_p\overline t}\,\frac{\partial f}{\partial t} \,+\,\vv\cdot\nabla_{\xx}f\,+\,\left(\bE\,+\,\frac{\omega_c}{\omega_p}
\,\vv\times \bB_{\rm ext}\right)\cdot\nabla_{\vv} f \,=\, 0.
\end{equation}
When the external magnetic field is assumed to be large, 
the rotation period of the electrons about the magnetic field
lines becomes small.  We introduce the
dimensionless cyclotron period 
$$
\varepsilon \,=\,\frac{\omega_p}{\omega_c}
$$
and since we are interested in asymptotically large time scale, we also have that 
$$
\varepsilon \,=\,\frac{1}{\overline t\,\omega_p}
$$
Then, under this scaling, the Vlasov equation \eqref{primes} for
$f=f_\varepsilon$ takes the form:
\begin{equation}
\label{V_varepsilon_f0}
\varepsilon\,\frac{\partial f^\varepsilon}{\partial t}\,+\,\vv \cdot \nabla_\xx f^\varepsilon \,+\,
\left(\bE\,+\,\frac{1}{\varepsilon}\vv\times \bB_{\rm ext}\right) \cdot \nabla_\vv f^\varepsilon\,=\,0,
\end{equation}
with initial conditions still given by (\ref{VIC}). 

\subsection{Assumptions and main result}
\label{ssec_main}

To simplify the presentation and the following calculation, we assume
that $\Omega=\RR^3$ and the external magnetic field only applies in the $z$-direction 
\begin{equation}
\label{hyp:B0}
\bB_{\rm ext}(t,\xx) = (0,0,b(t,\xx_\perp))^t
\end{equation}
where $\xx=(x,y,z)^t\in\RR^3$ with $\xx_\perp=(x,y)$ and
$x_\parallel=z$. The velocity variable will be denoted in the same
manner $\vv=(\vv_\perp,v_\parallel)$, with $\vv_\perp=(v_x,v_y)$ and
$v_\parallel=v_z$. 

Since the external magnetic field must satisfy the Gauss's law
for magnetism
$$
\nabla_\xx\cdot \bB_{\rm ext} = 0,
$$
it gives that indeed $b$ only depends on $x_\perp\in\RR^2$ and $t\in\RR^+$. Furthermore, we assume that $b$ does not vanish and is smooth: there
exists $\alpha>0$ such that 
\begin{equation}
\label{hyp:B1}
b\in W^{1,\infty}(\RR^+\times\RR^2), \quad  b(t,\xx_\perp) > \alpha.
\end{equation}

Under these assumptions, the Vlasov equation \eqref{V_varepsilon_f0} can
be written in a simple form, which allows us to ignore curvature effects 
\begin{equation}
\label{V_varepsilon_f}
\left\{
\begin{array}{l}
\displaystyle \varepsilon\,\frac{\partial f^\varepsilon}{\partial t}\,+\,\vv \cdot \nabla_\xx f^\varepsilon \,+\,
\left(\bE^\varepsilon\,+\,\frac{b}{\varepsilon}\vv^\perp\right) \cdot \nabla_\vv f^\varepsilon\,=\,0,
\\
\,
\\
\ds -\Delta \phi^\varepsilon = \rho^\varepsilon = \int_{\RR^3} f^\varepsilon d\vv, \quad \bE^\varepsilon = -\nabla_\xx
\phi^\varepsilon,
\\
\,
\\
f^\varepsilon(0) \,=\, f^\varepsilon_{\rm in}, 
\end{array}\right.
\end{equation}
where the operator $\,^\perp$ corresponds to a rotation of $-\pi/2$
along the axis $(0z)$ and therefore it only acts on
the $\vv_\perp=(v_x,v_y)$ component and keeps the third component
identical :  for
any $\vv=(\vv_\perp,v_\parallel)\in\RR^3$, we have
$\vv^\perp=(v_y,-v_x,v_z)$.

Let us first emphasize that applying the arguments  of A.A. Arsen'ev
\cite{arsenev_4} and
R. DiPerna and P.-L. Lions \cite{lions0_4}, we easily prove the existence of weak solutions  for any
$\varepsilon >0$.

\begin{theorem}
\label{th:0}
Assume the magnetic field satisfies (\ref{hyp:B0})-(\ref{hyp:B1}) and
the initial datum $f_{\rm in}^\varepsilon$  is a  nonnegative function such that
\begin{equation}
\label{hyp:f1}
f_{\rm in}^\varepsilon \in L^1\cap L^{\infty}(\RR^3\times \RR^3), \quad
\frac{\xx}{|\xx|^3}\star\int_{\RR^3} f_{\rm in}^\varepsilon (\xx,\vv)\,d\vv \in L^2(\RR^3),
\end{equation}
and has  finite kinetic energy
\begin{equation}
\label{hyp:f2}
\frac{1}{2}\int_{\RR^3\times\RR^3}
{\|\vv\|^2}\,f_{\rm in}^\varepsilon\,d\xx\,d\vv < \infty.
\end{equation}
Then there is a weak solution $(f^\eps,\bE^\eps)$ to the Vlasov-Poisson system
(\ref{V_varepsilon_f}), where $f^\eps \in
L^{\infty}(\RR^+, L^1\cap L^{\infty}(\RR^6))$,  the charge densities $\rho^\eps$ is such that 
$$
\rho^\eps\in L^{\infty}(\RR^+; L^{5/3}(\RR^3))
$$
and 
$$
\bE^\eps\in L^{\infty}(\RR^+;L^2\cap W^{1,5/3}(\RR^3)).
$$
\end{theorem} 

Note here that the $L^p$ and energy estimates  hold uniformly with
respect to $\varepsilon>0$. Furthermore, this results strongly relies
on the energy estimate, which is uniform with respect to $\eps>0$. We
define the total energy associated to \eqref{V_varepsilon_f}
$$
\mathcal{E}^\eps(t) := \int_{\RR^6} \frac{|\vv|^2}{2}\, f^\eps(t)
\,d\xx\,d\vv \;\;+\;\;   \frac{1}{2}\,\int_{\RR^3} |\bE^\eps|^2 \,d\xx
\,\,\leq\,\,\mathcal{E}^\eps(0).
$$

The aim of this paper is then to obtain a systematic expansion of
Hilbert type of the function  $(f^\varepsilon,\bE^\varepsilon)$
solution to the Vlasov-Poisson system \eqref{V_varepsilon_f} and to study the asymptotic model formally obtained by taking the limit
$\varepsilon \to 0$. 


Let us assume that $(f^\varepsilon,\bE^\varepsilon)$ can be written as 
\begin{equation}
\label{exp:01}
f^\varepsilon \,=\, \sum_{k\in\NN} \varepsilon^k \, f_k,
\quad  \bE^\varepsilon \,=\, \sum_{k\in\NN} \varepsilon^k \, \bE_k,
\end{equation}
where for any $k\in\NN$, $f_k$ and $\bE_k$ do not depend on
$\varepsilon$. The existence of such an expansion would guarantee that
$f^\varepsilon$ and $\bE^\varepsilon$ and their derivatives with
respect to $\xx$ and $\vv$ are uniformly bounded, at least if the
functions $f_k$, $\bE_k$ are sufficiently smooth. 

In particular we assume   that 
\begin{equation}
\label{hyp:f3}
f^\varepsilon_{\rm in}  = \sum_{k\in\NN} \varepsilon^k \, f_{{\rm
    in}, k},
\end{equation}
where for any $k\in\NN$, $f_{{\rm in}, k}$ does not depend on $\varepsilon>0$.

To introduce the gyroaveraging operator in the orthogonal plane to
magnetic field,  we will work in polar
coordinate  for $\vv_\perp=(v_x,v_y)\in\RR^2$, 
\begin{equation}
\label{c:polar}
\left\{
\begin{array}{l}
\ds v_x \, =\,  w \,\cos(\theta),
\\
\ds v_y \, =\,  w \,\sin(\theta),
\end{array}\right.
\end{equation}
and set   $\ee_w = (\cos\theta,\sin\theta)$, $\ee_{\theta}=-\ee_w^\perp=(-\sin\theta,\cos\theta)$.
Then we introduce the gyroaveraging operator
$\Pi$ defined for every function $f(\vv)$ by
\begin{equation}\label{def_gyromoy}
\Pi f\,(w,v_\parallel) \,=\,\frac1{2\pi} \int_{0}^{2\pi}  f(\vv_\perp,
v_\parallel) d\theta,
\end{equation}
where in the integral $\vv_\perp$ is expressed thanks to the change of
coordinate (\ref{c:polar}). 

\begin{theorem}[Formal expansion of $f ^\varepsilon$ and $\bE ^\varepsilon$]
\label{th:01}
 Let us consider an external the magnetic field  such that
 (\ref{hyp:B0})-(\ref{hyp:B1}) and   $f_{\rm in}^\varepsilon$  a  nonnegative function
satisfying (\ref{hyp:f1}), (\ref{hyp:f2}) and (\ref{hyp:f3}). Assume
there  exists a sequence $(f_k,\bE_k)_{k\in\NN}$ such that the
weak solutions $(f^{\eps},\bE^{\eps})_\eps$  to the Vlasov-Poisson system
(\ref{V_varepsilon_f}) can be expanded as (\ref{exp:01}) for
all $\varepsilon>0$. 
Then,  
$$
\left\{
\begin{array}{lll}
\bE_0&\equiv& \bE_F(t,\xx),
\\
f_0&\equiv& F(t,\xx,w,v_\parallel)
\end{array}\right.
$$ 
and there exists $(P,\bE_P)$ with $P\equiv P(t,\xx,w,v_\parallel)$
such that   $(F,P,\bE_F,\bE_P)$ is solution to the
following  system 
\begin{equation}
\label{transport:F}
\left\{
\begin{array}{l}
\ds\frac{\partial F}{\partial t}\,+ \,
\bU_\perp\cdot\nabla_{\xx_\perp}F
\,+\,u_w\,\frac{\partial
  F}{\partial w} \,-\,  \frac{\partial \phi_P}{\partial x_\parallel} \frac{\partial F}{\partial
  v_\parallel} 
\,+\,  v_\parallel\frac{\partial P}{\partial x_\parallel}
\,-\,  \frac{\partial \phi_F}{\partial x_\parallel} \frac{\partial P}{\partial v_\parallel}\,=\,0,
\\
\, 
\\
\ds v_\parallel \frac{\partial F}{\partial x_\parallel} 
\,-\,  \frac{\partial \phi_F}{\partial x_\parallel} \frac{\partial F}{\partial v_\parallel} = 0,
\\
\, 
\\
\ds F(0) = \Pi f_{{\rm in}, 0}, \,\, P(0) = \Pi f_{{\rm in}, 1},
\end{array}\right.
\end{equation}
where $\bU_\perp$ corresponds to the drift velocity and
$(\bU_\perp,u_w)$ is given by
\begin{equation}
\label{UperpW}
\bU_\perp = -\frac{1}{b} \left( \nabla_{\xx_\perp}\phi_F \,+\, \frac{w^2}{2b}
  \nabla_{\xx_\perp} b  \right)^\perp, 
\quad
u_w\,=\, \frac{w}{2b^2}\,\nabla_{\xx_\perp}^\perp b\cdot\nabla_{x_\perp}\phi_F,  
\end{equation}
and the electric fields $\bE_F$, $\bE_P$ are such that $\bE_Q=-\nabla\phi_Q$, for $Q=F,P$
and $\phi_Q$ solves the Poisson equation with source terms $\rho_Q$,
\begin{equation}
\label{poisson:i}
\ds -\Delta \phi_Q = \rho_Q:= 2\pi\int_{\RR^+\times\RR} Q(t,\xx,w,v_\parallel) \,w\,dw\, dv_\parallel, \quad Q=F,P.
\end{equation}
Moreover, the following relation holds
$$
f_1 =  -\frac{1}{b(t,\xx_\perp)}\,\, {\bf e}_\theta\cdot
  \left( w\nabla_{\xx_\perp} F \,+\, \bE_\perp \frac{\partial F}{\partial w} \right)  \,\,+\,\, P(t,\xx,r,v_\parallel).
$$
\end{theorem}

Let us emphasize that the first equation of \eqref{transport:F} means that $F$ and $P$ do not depend 
of the gyrophase $\theta\in [0,2\pi]$, but on $w=\|\vv_\perp\|$. In this model, $F$ is
determined by  \eqref{transport:F} while the unknown function $P$
plays the role  of the Lagrange multiplier associated to the
constraint in \eqref{transport:F}. This constraint reflects the fact that the fast parallel motion along the magnetic
field line is  instantaneously relaxed and $F$ is constant along these
trajectories. In other words, there are three time scales for a
particle moving in a large magnetic field:
\begin{itemize}
\item the fastest time scale corresponds to
the cyclotron or Larmor rotation period about the magnetic field. This
time scale is eliminated here by averaging over $\theta\in (0,2\pi)$;

\item  the second fastest scale  is the scale of the parallel motion along the magnetic
field line, which is described here by the constraint in
(\ref{transport:F});

\item the slow time scale corresponds to the various drifts across the magnetic field lines, due to spatio-temporal
variations of the electromagnetic field. In the
system (\ref{transport:F}), we focus on the slow time scale, which
corresponds to the large time behavior of the solution to the
Vlasov-Poisson system  (\ref{V_varepsilon_f}). 
\end{itemize}

Of course these various drifts are often obtained directly on the
particle trajectories, but the averaging effect is difficult to
justify and is only valid  for slowly varying electromagnetic
fields. The use of the kinetic model directly provides a way to do it
by imposing constraints on the distribution function. This easier derivation reflects the fact that, to some extent, the
distribution function describes the particle dynamics in a statistical
sense. Averaging the trajectories over some fast component is best done by looking at the
evolution of an observable of the system which is constant over this fast motion.


\section{Fundamental properties of the asymptotic model}
\label{sec:2.bis}

In this section, we prove some fundamental properties satisfied by the
asymptotic model (\ref{transport:F}), which illustrates the physical validity of the present
approach. In the following, we show that
\begin{itemize}
\item we recover the classical drift velocity $\bE\times \bB$ and the
  gradient drift velocity;
\item the asymptotic model satisfies conservation of energy;
\item the magnetic moment   is an invariant of the asymptotic model. 
\end{itemize}

\subsection{Drift velocities}
The drift velocity $\bU_\perp$ corresponds to the sum of classical  guiding
center drift $\bU_{\bE}$ and  $\bU_{\nabla\bB}$. Indeed,  the external
magnetic field only acts on the $z$-direction, $\bB_{\rm
  ext}=(0,0,b)$, then we first recover the drift velocity called $\bE\times \bB$ 
$$
\bU_{\bE} \,\,=\,\, \frac{\bE \times\bB_{\rm ext}}{\|\bB_{\rm ext}\|^2} \,\,=\,\,
-\frac{\nabla_{\xx_\perp}^\perp\phi_F}{b}.  
$$ 
and the so called gradient-$B$ drift,
$$
\bU_{\nabla\bB} \,\,=\,\, \frac{\|\vv_\perp\|^2}{2}\,\, \frac{\bB_{\rm ext}\times
  \nabla\left|\bB_{\rm ext}\right|}{\|\bB_{\rm ext}\|^3} \,\,=\,\, -\frac{w^2}{2}\,
\frac{\nabla^\perp_{\xx_\perp} b}{b^2}. 
$$
Here there is no curvature drift since we considered this simple
external magnetic field $\bB_{\rm
  ext}=(0,0,b)$.

\subsection{$L^p$ norms and total energy conservation}

Let us first write (\ref{transport:F})  in a conservative form
\begin{eqnarray}
\nonumber
&&\frac{\partial}{\partial t}\left( w\,F\right)\,+ \,
\nabla_{\xx_\perp}\cdot\left( w\,\bU_\perp \,F\right)
\,+\, \frac{\partial}{\partial w}\left(w\,u_w \,F\right) \, \,-\,\,  \frac{\partial }{\partial
  v_\parallel} \left(\frac{\partial\phi_P}{\partial{x_\parallel}} \,w\,F\right)
\\
&&\,+\,  \frac{\partial}{\partial x_\parallel}\left( w\,v_\parallel \,P\right)\,-\,
\frac{\partial }{\partial v_\parallel}\left( w\,\frac{\partial\phi_F}{\partial{x_\parallel}}  \,P\right)\,=\,0.
\label{conserv:F}
\end{eqnarray}
Thus, we prove the following conservation property. 
\begin{proposition}
\label{prop:3.1}  
Let  $f_{\rm in}^\varepsilon$  be a  nonnegative function satisfying
(\ref{hyp:f1})-(\ref{hyp:f2}) and (\ref{hyp:f3}). Assume that the
limiting system  (\ref{transport:F}) has a smooth solution $(F,\bE_F)$
and $(P,\bE_P)$. Then,  for any $\Theta\in \mathcal{C}^1(\RR)$ such that
$$
\int_{\RR^3\times\RR^+\times\RR} \Theta(F(0)) \,w\,dw\,dv_\parallel \,d\xx\,<\,+\infty,
$$
we have
$$
\int_{\RR^3\times\RR^+\times\RR} \Theta(F(t)) \,w\,dw\,dv_\parallel \,d\xx \,\;=\;\, \int_{\RR^3\times\RR^+\times\RR} \Theta(F(0)) \,w\,dw\,dv_\parallel \,d\xx.
$$ 
\end{proposition}
\begin{proof}
Assuming that $F$ is a smooth solution to (\ref{transport:F}) together
with  a smooth $P$, we multiply (\ref{transport:F}) by $\Theta^\prime(F)$ and integrate with
respect to $(\xx,w,v_\parallel)\in\RR^3 \times \RR^+\times\RR$. Then,
using that
$$
w\,\nabla_{\xx_\perp}\cdot\bU_\perp \,\,=\,\,
-\frac{w}{b^2}\,\nabla_{\xx_\perp}^\perp b \cdot \nabla_{\xx_\perp}\phi_F
$$
and observing that
$$
v_\parallel \frac{\partial \Theta^\prime(F)}{\partial x_\parallel} \,-\,
\frac{\partial \phi_F}{\partial x_\parallel} \,\frac{\partial \Theta^\prime(F)}{\partial v_\parallel} = 0,
$$
we easily get after a simple integration by part and for suitable
boundary conditions (either periodic or vanishing property in the whole space)
$$
\frac{d}{dt} \int_{\RR^3\times\RR^+\times\RR} \Theta(F(t))
\,w\,dw\,dv_\parallel \,d\xx \,=\, 0,
$$ 
hence the results follows.
\end{proof}

From  Proposition \ref{prop:3.1}, we get classical $L^p$ estimates, for any
$1\leq p\leq \infty$, on the
distribution function $F$ :  for all $t\in\RR^+$ 
\begin{equation}
\label{Lp:F}
\|F(t)\|_{L^p} \,\leq\, \|F_0\|_{L^p} 
\end{equation}
and $F(t)\geq 0 $ for any nonnegative initial data $F_0$.
\\

Another remarkable property of our model is the preservation of the
energy structure. Let us define  the total energy $\mathcal{E}^0(t)$:
for a smooth solution $(F,\bE_F)$ to  (\ref{transport:F}),
$$
\mathcal{E}^0(t):= \int_{\RR^3\times\RR^+\times\RR} 
\left(\frac{w^2+v_\parallel^2}{2}\right) \,F(t)\,w\,dw\,dv_\parallel\, d\xx \,\,+\,\, \frac{1}{4\pi}\int_{\RR^3}
\left| \bE_F(t) \right|^2 d\xx 
$$

\begin{proposition}
\label{prop:3.2}
Let  $f_{\rm in}^\varepsilon$  be a  nonnegative function satisfying
(\ref{hyp:f1})-(\ref{hyp:f2}) and (\ref{hyp:f3}). Assume that the
limiting system  (\ref{transport:F}) has a smooth solution $(F,\bE_F)$
and $(P,\bE_P)$. Then, we have for any $t\geq 0$,
$$
\mathcal{E}^0(t)\,\,\leq\,\, \mathcal{E}^0(0).
$$
\end{proposition}
\begin{proof}
Let us first multiply (\ref{conserv:F}) by
$|v_\parallel|/2$ and integrate both in space and velocity, it gives
$$
\frac{d}{dt} \int \frac{|v_\parallel|^2}{2}\, F(t) \,w\,dw\,dv_\parallel
\,d\xx \,=\, -\int \frac{\partial \phi_P}{\partial x_\parallel} \,v_\parallel \,F(t) \,w\,dw\,dv_\parallel
\,d\xx  \,-\, \int \frac{\partial \phi_F}{\partial x_\parallel} \,v_\parallel \,P(t) \,w\,dw\,dv_\parallel
\,d\xx.  
$$
Using the constraint in the parallel direction
(\ref{transport:F}) on $F$, we obtain the following cancellation
\begin{eqnarray*}
\int \frac{\partial \phi_P}{\partial x_\parallel} \,v_\parallel \,F(t) \,w\,dw\,dv_\parallel\,d\xx
&=& -\int \phi_{P} \,v_\parallel \,\frac{\partial F}{\partial
  x_\parallel}(t) \,w\,dw\,dv_\parallel\,d\xx  
\\
&=& \int \frac{\partial}{\partial
  v_\parallel} \left( \phi_{P} \,\frac{\partial \phi_F}{\partial x_\parallel} \,F(t)\right) \,w\,dw\,dv_\parallel\,d\xx  \,=\, 0,
\end{eqnarray*}
hence it yields
\begin{equation}
\label{ener:1}
\frac{d}{dt} \int \frac{|v_\parallel|^2}{2}\, F(t) \,w\,dw\,dv_\parallel
\,d\xx \,=\, -\int \frac{\partial \phi_F}{\partial x_\parallel} \,v_\parallel \,P(t) \,w\,dw\,dv_\parallel
\,d\xx.  
\end{equation}
Then we multiply (\ref{conserv:F}) by
$w^2/2$ and integrate both in space and velocity, we have
\begin{equation}
\label{ener:2}
\frac{d}{dt} \int \frac{w^2}{2}\, F(t) \,w\,dw\,dv_\parallel
\,d\xx \,=\, \int  \nabla_{\xx_\perp}^\perp b\cdot \nabla_{\xx_\perp} \phi_F
  \,F\, \frac{w^3}{2b^2}\,dw\,dv_\parallel\,d\xx.  
\end{equation}
Finally, we multiply (\ref{conserv:F}) by the potential $\phi_F$
computed from $\rho_F$ and after integration, we have 
 \begin{eqnarray*}
\int \frac{\partial F}{\partial t} \phi_F
 \,w\,dw\,dv_\parallel\,d\xx &=& \int \nabla_{\xx_\perp}
 \phi_F \cdot \bU_\perp F \,w \,dw \,dv_\parallel\, d\xx  \,+\, 
 \int \frac{\partial \phi_F}{\partial x_\parallel} v_\parallel \,P(t)
 \,w\,dw \,dv_\parallel \,d\xx 
\\
&=& -\int \nabla_{\xx_\perp}
 \phi_F \cdot  \nabla_{\xx_\perp}^\perp b\,  F \frac{w^3}{2b^2} \,dw \,dv_\parallel\, d\xx  \,+\, 
 \int \frac{\partial \phi_F}{\partial x_\parallel} v_\parallel \,P(t)
 w\,dw \,dv_\parallel \,d\xx. 
\end{eqnarray*}
On the other hand, by definition of $\rho_F$ and using the Poisson
equation (\ref{poisson:i}), we have
\begin{eqnarray*}
\int \frac{\partial F}{\partial t} \phi_F
 \,w\,dw\,dv_\parallel\,d\xx &=& \frac{1}{2\pi}  \int  \frac{\partial
   \rho_F}{\partial t} \,\phi_F d\xx
\\
&=& \frac{1}{4\,\pi} \frac{d}{dt} \int \left|\nabla_\xx \phi_F\right|^2 d\xx.
\end{eqnarray*}
Hence, gathering the later results, we obtain
\begin{eqnarray}
\label{ener:3}
\frac{1}{4\,\pi} \frac{d}{dt} \int \left|\nabla_\xx \phi_F\right|^2
d\xx &=& -\int \nabla_{\xx_\perp}\phi_F \cdot
\nabla_{\xx_\perp}^\perp b\,  F \frac{w^3}{2b^2} \,dw \,dv_\parallel\,
d\xx  
\\
&+& 
\nonumber
 \int \frac{\partial \phi_F}{\partial x_\parallel}\, v_\parallel \,P(t)
 w\,dw \,dv_\parallel \,d\xx. 
\end{eqnarray}
Finally, adding (\ref{ener:1}), (\ref{ener:2}) and (\ref{ener:3}), we
get the conservation of energy 
$$
\frac{d}{dt} \mathcal{E}^0(t) = 0. 
$$
\end{proof}

\subsection{Invariance of the magnetic moment (first adiabatic invariant)}
In this section we assume that the external magnetic field does not
depend on time. Then we define $\mu$ as the magnetic moment 
$$
\mu = \frac{w^2}{2\,b(\xx_\perp)} 
$$
and let us show that it is an invariant of the movement for the
asymptotic model (\ref{transport:F}).

We compute the time derivative along the flow,
$$
\frac{d\mu}{dt} \,=\, \frac{w}{b} \,\frac{dw}{dt}  \,-\, \frac{w^2}{2b^2}
\,\nabla_{\xx_\perp}b\cdot \frac{d\xx_\perp}{dt}.
$$
Using the characteristic curves to (\ref{transport:F}) and the
orthogonality properties of the $\,^\perp$ operator, it yields
\begin{eqnarray*}
\frac{d\mu}{dt} &=& \frac{w^2}{2b^3}\,\nabla_{\xx_\perp}^\perp b\cdot \nabla_{\xx_\perp}\phi_F   \,+\, \frac{w^2}{2b^3}
\,\nabla_{\xx_\perp}b\cdot  \left( \nabla_{\xx_\perp}\phi_F +\frac{w^2}{2}
  \frac{\nabla_{\xx_\perp} b}{b}  \right)^\perp,
\\
&=& \frac{w^2}{2b^3}\,\nabla_{\xx_\perp}^\perp b\cdot \nabla_{\xx_\perp}\phi_F  \,+\, \frac{w^2}{2b^3}
\,\nabla_{\xx_\perp}b\cdot \nabla_{\xx_\perp}^\perp\phi_F,
\\
&=& 0.
\end{eqnarray*}
Therefore, we can perform a change of variable on (\ref{transport:F})
to get the time evolution of the distribution function expressed in
term of the magnetic moment $F\equiv F(t,\xx,\mu,v_\parallel)$, it
yields the following equation 
\begin{equation}
\label{transport:Fmu}
\left\{
\begin{array}{l}
\ds\frac{\partial F}{\partial t}\,+ \,
\bU_\perp\cdot\nabla_{\xx_\perp}F\,-\,  \frac{\partial \phi_P}{\partial x_\parallel} \frac{\partial F}{\partial v_\parallel} 
\,+\,  v_\parallel\frac{\partial P}{\partial x_\parallel}\,-\,
\frac{\partial \phi_F}{\partial x_\parallel}\frac{\partial P}{\partial v_\parallel}\,=\,0,
\\
\, 
\\
\ds v_\parallel \frac{\partial F}{\partial x_\parallel} \,-\,
\frac{\partial \phi_F}{\partial x_\parallel} \frac{\partial F}{\partial v_\parallel} = 0,
\end{array}\right.
\end{equation}
where $\bU_\perp$ is now given by 
$$
\bU_\perp \,=\, -\frac{1}{b} \left( \nabla_{\xx_\perp}\phi_F \,+\, \mu  \nabla_{\xx_\perp} b  \right)^\perp, 
$$
 and the electric fields
$\bE_F$, $\bE_P$ in (\ref{poisson:i}). Notice that since $b(\xx_\perp)\,d\mu
\,d\xx_\perp= w\,dw \,d\xx_\perp$, the limiting system (\ref{transport:Fmu}) can be
written in conservative form when $b$ does not vary with time
 
\begin{equation}
\label{conserv:Fmu}
\left\{
\begin{array}{l}
\ds\frac{\partial bF}{\partial t}\,+ \,
\nabla_{\xx_\perp}\cdot \left(\bU_\perp\,bF\right)\,-\,\frac{\partial}{\partial v_\parallel} \left(\frac{\partial \phi_P}{\partial x_\parallel} \,bF\right)
\,+\,  \frac{\partial }{\partial x_\parallel}\left(v_\parallel\,bP\right)\,-\,
\frac{\partial}{\partial v_\parallel}\left( \frac{\partial \phi_F}{\partial x_\parallel}\,bP\right)\,=\,0,
\\
\, 
\\
\ds\frac{\partial}{\partial x_\parallel}\left(  v_\parallel \,bF\right) \,-\,
\frac{\partial
  }{\partial v_\parallel}\left( \frac{\partial \phi_F}{\partial x_\parallel} bF\right) \,=\, 0.
\end{array}\right.
\end{equation}

\subsection{Longitudinal invariant (second adiabatic
  invariant)}

From the constraint in (\ref{transport:F}), we easily deduce after
integration over $v_\parallel\in \RR$ that
$$
\frac{\partial }{\partial x_\parallel} \int_{\RR} v_\parallel F
dv_\parallel = \frac{\partial \phi_F}{\partial x_\parallel} \int_{\RR}
\frac{\partial F} {\partial v_\parallel}  dv_\parallel = 0,
$$ 
which gives the conservation of momentum along the magnetic field
line.


\section{The asymptotic limit $\varepsilon \to 0$}
\label{sec:3} 

It is worth to mention here that these {\it a priori} estimates on the
distribution function $F$ and the electric field $\bE_F$ would give enough 
compactness to treat the nonlinear term $\bU_\perp b F$, but we do not
get any estimate on the additional term $(P,\bE_P)$ so that the
existence of weak solution to the limiting system (\ref{transport:F})
is still an open problem.

In order to establish a convergence result, we apply a formal analysis
of the system (\ref{V_varepsilon_f}). Applying a standard  Hilbert
expansion to $(f^\varepsilon, \bE^\varepsilon)$ solution to the
Vlasov-Poisson system (\ref{V_varepsilon_f}), we get a
hierarchy of differential equations which have to be solved at each order. Here, we take advantage of the simple structure of the magnetic
field to solve explicitly each problem and get the asymptotic model
(\ref{transport:F}).
 

\subsection{The Hilbert expansion}
\label{ssec_Hilbert}
Since the leading order term in (\ref{V_varepsilon_f})  involves the effect of a circular motion around the magnetic field
lines, we now specifically examine the properties of this operator. Let us denote by $\mathcal{L}$ the following operator
$$
\mathcal{L} f \,\,=\,\,  -b(t,\xx_\perp) \,\,\vv^\perp \cdot\nabla_\vv f.
$$ 
We have the following result
\begin{lemma}
\label{lem:01} 
Assume that $b$ satisfies (\ref{hyp:B0})-(\ref{hyp:B1}). Then, the null space $\ker \mathcal{L}$ of $\mathcal{L}$ consists of functions which only depend
on the parallel component $v_\parallel$ and on the amplitude of
$\vv_\perp$, that is, $w = \|\vv_\perp\|$, 
\begin{equation}
\mathcal{L} f \, =\,  0 \,\,  \Longleftrightarrow \, \, f(\vv) \,
\equiv \,   \tilde f (w,v_\parallel) \, \mbox{ with } \, (w,v_\parallel) \in \RR^+\times\RR. 
\label{null_space_L}
\end{equation}
\end{lemma}

\begin{proof}
On the one hand, we notice that the magnetic field $\bB_{\rm ext}$
only acts on $\vv_\perp=(v_x,v_y)$, which means that
$$
\mathcal{L} f \, \, =\, \,  -b(t,\xx_\perp) \, \, \vv^\perp_\perp \, \cdot \, \nabla_{\vv_\perp} f,
$$ 
where now $\,^\perp$ acts on a vector of $\RR^2$ and
$\vv_\perp^\perp=(v_y,-v_x)$.  Then,  applying a change of variable
to polar coordinates on $\vv_\perp\in\RR^2$, it yields
$$
\mathcal{L} f  \, \,=\, \, b(t,\xx_\perp) \, \,\frac{\partial f}{\partial \theta}.
$$
From (\ref{hyp:B1}), the magnetic field does not vanish, hence we get
$\mathcal{L} f =0$ if and only if $f(\vv) \equiv \tilde f (w,v_\parallel)$, which proves (\ref{null_space_L}). 
\end{proof}

Now, our goal is to find the asymptotic limit  $\varepsilon \to 0$ to
the Vlasov-Poisson system  \eqref{V_varepsilon_f}.  We start by assuming that $(f^\varepsilon, \bE^\varepsilon)$ admits an Hilbert expansion: 
$$
f^\varepsilon \,\,=\,\,  f_0+ \varepsilon \, f_1 + \varepsilon^2 \,
f_2 + \ldots,
$$
and
$$ 
\bE^\varepsilon \,\,=\,\, \bE_0 + \varepsilon\, \bE_1 + \varepsilon^2\, \bE_2 + \ldots
$$
Inserting these expansions in the Vlasov-Poisson system \eqref{V_varepsilon_f}, we find
for the leading order $\varepsilon^{-2}$, $\varepsilon^{-1}$ and $\varepsilon^{0}$ that 
\begin{eqnarray}
\label{order-2}
\mathcal{L} f_0 &=& 0. 
\\
\label{order-1}
\mathcal{L} f_1 &=& \vv\cdot \nabla_\xx f_0
\,+\,  \bE_0\cdot \nabla_\xx f_0.
\\
\label{order-0}
\mathcal{L} f_2 &=& \frac{\partial
  f_0}{\partial t} \,+\,\vv\cdot \nabla_\xx f_1
\,+\,  \bE_0\cdot \nabla_\vv f_1 \,+\,  \bE_1\cdot \nabla_\vv f_0.
\end{eqnarray}  

In order to solve eq. $\mathcal{L} f  = h$, we proceed in two steps : 
\begin{itemize}
\item we verify the solvability condition  $\Pi\, h=0$;
\item we compute $f$ by integrating $h$ over $\theta$.
\end{itemize}

\subsection{Proof of Theorem \ref{th:01}}

In this section we derive an asymptotic model for the limit $f_0$ of $f^\varepsilon$ by formally passing to the limit
$\varepsilon\to0$ in the Vlasov-Poisson system \eqref{V_varepsilon_f}.  This model will be
deduced by solving the sequence of equations appearing in the Hilbert
expansion \eqref{order-2}-\eqref{order-0}. 

First, by a simple application of Lemma \ref{lem:01},  the leading
order of the Hilbert expansion (\ref{order-2}) can be directly
solved. The function $f_0$ does not depend on $\theta\in[0,2\pi]$ and $f_0\equiv F(t,\xx,w, v_\parallel)$  for any
$\vv=(\vv_\perp, v_\parallel)\in\RR^3$ and at time $t=0$, we set
$F(0) =  \Pi f_{{\rm in}, 0}$, where $f_{{\rm in}, 0}$ is given from
the expansion of the initial data
$f^\varepsilon_{\rm in}$ in (\ref{hyp:f3}). 

Moreover, substituting the Hilbert expansion to $\bE^\eps$ in the Poisson
equation in (\ref{V_varepsilon_f}), gives that  $\bE_0=\bE_F:=-\nabla \phi_F$ with 
$$
\Delta \phi_F \,=\, \rho_F= 2\pi \int_{\RR^+\times\RR} F \,w\,dwdv_\parallel.
$$

Now, the goal is to find the equation satisfied by $F$, hence we turn 
to~(\ref{order-1}) and first  set 
\begin{equation}
\label{def:G}
\bG(t,\xx, w, v_\parallel) \,\,:=\,\, w \,\nabla_{\xx_\perp}F\, -\, \nabla_{\xx_\perp}\phi_F
\,\frac{\partial F}{\partial w}.
\end{equation}
Then we  prove the following Proposition.

\begin{proposition}
\label{prop:4.3}
Assume that $b$ satisfies (\ref{hyp:B0})-(\ref{hyp:B1}) and consider
$(F,\bE_F)$ the leading order of the Hilbert expansion (\ref{exp:01}).  Then equation \eqref{order-1} admits a solution $f_1$ if and only if $F$
satisfies the solvability condition 
$$
 v_\parallel \frac{\partial F}{\partial x_\parallel} \,-\,
\frac{\partial \phi_F}{\partial x_\parallel}  \frac{\partial F}{\partial v_\parallel} = 0.
$$  
Moreover, if this condition is satisfied, then there exists a function $P\in
\ker\mathcal{L}$  such that
$$
\left\{
\begin{array}{lll}
\ds f_1(t,\xx,\vv) &=& -\frac{1}{b(t,\xx_\perp)}\, {\bf e}_\theta\cdot
  \bG(t,\xx,r,v_\parallel)  \,\,+\,\, P(t,\xx,r,v_\parallel),
\\
\,
\\
\bE_1(t,\xx) &=& \bE_P:=-\nabla \phi_P,
\end{array}\right. 
$$
with ${\bf e}_\theta=(-\sin\theta,\cos\theta)^t$, $\phi_P$ satisfies
the Poisson equation (\ref{poisson:i}) and at time $t=0$, we
have $P(0) = \Pi f_1(0,\xx,\vv)$. 
\end{proposition}
\begin{proof}
Thanks to the definition of $\bG$, we  write (\ref{order-1})  as
$$
\mathcal{L} f_1 \, =\,  {\bf e}_w\cdot  \bG(t,\xx,w,v_\parallel) \,+\,
v_\parallel \frac{\partial F}{\partial x_\parallel} \,-\,
\frac{\partial \phi_F}{\partial x_\parallel} \frac{\partial F}{\partial v_\parallel}.
$$

On the one hand, we require that the solvability condition of
(\ref{order-1})  is well satisfied 
$$
\Pi \,\mathcal{L} f_1 \,\,=\,\, 0.
$$
Since $\Pi {\bf e}_w=0$, a necessary and sufficient condition for the solvability of
(\ref{order-1}) is that $F$ satisfies the following condition
\begin{equation*}
v_\parallel \frac{\partial F}{\partial x_\parallel} \,-\,
\frac{\partial \phi_F}{\partial x_\parallel}  \frac{\partial F}{\partial v_\parallel} = 0,
\end{equation*}
which corresponds to the constraint equation in (\ref{transport:F}). 

On the other hand, assuming that this solvability condition is
verified, we can  explicitly solve (\ref{order-1})   by integration with
respect to $\theta\in [0,2\pi]$. Then there exists a function
$P\in \ker\mathcal{L}$ such that for any $(t,\xx,\vv)\in\RR^+\times\RR^3\times\RR^3$, 
$$
f_1(t,\xx,\vv) = -\frac{1}{b(t,\xx_\perp)}\, {\bf e}_\theta\cdot
  \bG(t,\xx,w,v_\parallel)    \,\,+\,\, P(t,\xx,w,v_\parallel),
$$ 
where $\Pi \ee_\theta=0$  and from the initial condition
(\ref{hyp:f3}), it  gives that
$$
P(0) = \Pi f_{{\rm in}, 1}.
$$
Finally, substituting the Hilbert expansion to $\bE^\eps$ in the Poisson
equation in (\ref{V_varepsilon_f}) and using that 
$$
\Pi f_1 = \frac{1}{b}\Pi \ee_\theta \cdot \bG \,+\,  \Pi  P = \, \Pi  P, 
$$
we observe that 
$$
\rho_1 = \int_{\RR^3} f_1 d\vv =  2\pi \int_{\RR^+\times\RR} P \,w\,dwdv_\parallel,
$$
which gives that  $\bE_1=\bE_P:=-\nabla \phi_P$ with $\phi_P$ solution
to the Poisson equation
$$
-\Delta \phi_P \,=\, \rho_F= 2\pi \int_{\RR^+\times\RR} P \,w\,dwdv_\parallel.
$$

\end{proof}

Note that $f_1$ depends now on the whole variable
$(\xx,\vv)\in\RR^3\times\RR^3$. Finally, the equation satisfied by $F$ now appears as the solvability condition of (\ref{order-0}).

\begin{proposition}
\label{prop:4.4}
Assume that $b$ satisfies (\ref{hyp:B0})-(\ref{hyp:B1}) and consider
$(F,\bE_F)$ the leading order of the Hilbert expansion (\ref{exp:01}). Then, equation \eqref{order-0} admits a solution $f_2$ if and only if $F$
satisfies the first equation in (\ref{transport:F}), that is,
$$
\frac{\partial F}{\partial t}\,+ \, \bU_\perp\cdot\nabla_{\xx_\perp}F
\,+\,u_w  \frac{\partial
  F}{\partial w} \,-\,  \frac{\partial \phi_P}{\partial x_\parallel} \frac{\partial F}{\partial
  v_\parallel} 
\,+\,  v_\parallel\frac{\partial P}{\partial x_\parallel}\,-\,  \frac{\partial \phi_F}{\partial x_\parallel} \,\frac{\partial P}{\partial v_\parallel}\,=\,0,
$$
where the drift velocity $\bU_\perp$ and $u_w$
are given by (\ref{UperpW}).
\end{proposition}
\begin{proof}
As before we apply the solvability condition to  (\ref{order-0}),
which  corresponds to
$$
\Pi \,\mathcal{L} f_2 \,\,=\,\, 0,
$$
or it can be written as
\begin{equation}
\label{rat}
\int_0^{2\pi} \left(\frac{\partial F}{\partial t} \,+\,\vv\cdot \nabla_\xx f_1
\,+\,  \bE_F\cdot \nabla_\vv f_1 \,+\,  \bE_P\cdot \nabla_\vv F\right)
d\theta = 0,
\end{equation}
where $\bE_Q=-\nabla \phi_Q$ corresponds to electric field obtained by solving the
Poisson equation (\ref{poisson:i}). 

Let us compute explicitly each term  with respect to
$(F,\bE_F)$ and $(P,\bE_P)$ given from the previous analysis.
 
On the one hand, since the distribution function $F$ does not depend
on the angular velocity, we have that
$$
\frac{1}{2\pi}\int_0^{2\pi} \frac{\partial F}{\partial t} d\theta = \frac{\partial F}{\partial t}
$$
and 
$$
\frac{1}{2\pi}\int_0^{2\pi} \bE_P\cdot \nabla_\vv F \, d\theta =
- \frac{\partial \phi_P}{\partial x_\parallel} \frac{\partial F}{\partial v_\parallel}.
$$ 
On the other hand, from  the definition of $f_1$ given in Proposition \ref{prop:4.3}, we get that
\begin{equation*}
\frac{1}{2\pi}\,\nabla_\xx\cdot\left( \int_0^{2\pi} v f_1 d\theta\right)  \,=\,
-\frac{w}{2} \,\nabla_{\xx_\perp} \cdot\left( \frac{\bG^\perp}{b}\right)
\,+\,  v_\parallel\frac{\partial P}{\partial x_\parallel},
\end{equation*}
with 
$$
\bG^\perp =  w \,\nabla_{\xx_\perp}^\perp F\, -\, \nabla_{\xx_\perp}^\perp\phi_F
\,\frac{\partial F}{\partial w}.
$$
Then, since the operator $\nabla_{\xx_\perp}\cdot\nabla_{\xx_\perp}^\perp=0$, it
yields that
\begin{equation}
\label{rat:1}
\frac{1}{2\pi}\,\nabla_\xx\cdot\left( \int_0^{2\pi} v f_1 d\theta\right) \,=\, \frac{w}{2b}\,\nabla_{\xx_\perp}^\perp \phi_F\,\frac{\partial }{\partial
  w}\left(\nabla_{\xx_\perp} F\right)
\,+\,\frac{w}{2} \frac{\nabla_{\xx_\perp} b}{b^2}
 \,\cdot \bG^\perp \,+\, v_\parallel\frac{\partial P}{\partial x_\parallel}.
\end{equation}
Finally, we evaluate the penultimate term in (\ref{rat}). From the expression of $f_1$ in Proposition
\ref{prop:4.3},  we  obtain
\begin{equation*}
\frac{1}{2\pi}\,\bE_F\cdot  \left(\int_0^{2\pi} \nabla_\vv f_1 d\theta\right) \,=\, \frac{1}{2b}\,\nabla_{\xx_\perp}\phi_F  \cdot \left(
  \frac{\bG}{w} + \frac{\partial \bG}{\partial w}\right)^\perp \,-\,  \frac{\partial \phi_F}{\partial x_\parallel} \,\frac{\partial P}{\partial v_\parallel}.
\end{equation*}
Hence using the orthogonality property and  $\bu^\perp\cdot
\bw=-\bw^\perp\cdot \bu$ for any $(\bu,\bw)\in\RR^2\times\RR^2$, it yields
\begin{equation}
\label{rat:2}
\frac{1}{2\pi}\,\bE_F\cdot  \left(\int_0^{2\pi} \nabla_\vv f_1 d\theta\right)= -\frac{\nabla_{\xx_\perp}^\perp\phi_F}{b} \cdot \nabla_{\xx_\perp} F \,+\,
\frac{w}{2b}\,\nabla_{\xx_\perp} \phi_F\cdot\frac{\partial }{\partial
  w}\left(\nabla_{\xx_\perp}^\perp F\right) \,-\,  \frac{\partial \phi_F}{\partial x_\parallel} \,\frac{\partial P}{\partial v_\parallel}.
\end{equation}
Gathering (\ref{rat:1}) and (\ref{rat:2}), we  get some cancellation
and it yields to the following expression
\begin{eqnarray*}
\frac{1}{2\pi}\, \int_0^{2\pi}
  \left(\nabla_\xx\cdot v f_1\,+\, \bE_F\cdot \nabla_\vv f_1 \right) \,d\theta &=&
  \bU_\perp\cdot\nabla_{\xx_\perp} F \,+\, u_w\, \frac{\partial F}{\partial w}
\\ 
&+&  v_\parallel\frac{\partial P}{\partial x_\parallel} \,-\,  \frac{\partial \phi_F}{\partial x_\parallel} \,\frac{\partial P}{\partial v_\parallel}.
\end{eqnarray*}
where 
$$
\bU_\perp \,\,=\,\, -\frac{1}{b}\left(\nabla_{\xx_\perp}\phi \,+\,
  \frac{w^2}{2b}\nabla_{\xx_\perp} b\right)^\perp, \quad u_w\,=\, \frac{w}{2b^2}\,\nabla_{\xx_\perp}^\perp b\cdot\nabla_{x_\perp}\phi_F.
$$
Finally, the solvability condition on $f_2$ is satisfied once the
distribution function  $F$ is solution to the following equation
$$
\frac{\partial F}{\partial t} \,+ \, \bU_\perp\cdot\nabla_{\xx_\perp}F
\,+\, u_w \frac{\partial
  F}{\partial w} \,-\,  \frac{\partial \phi_P}{\partial x_\parallel} \, \frac{\partial F}{\partial
  v_\parallel} 
\,+\,  v_\parallel\frac{\partial P}{\partial x_\parallel}\,-\,  \frac{\partial \phi_F}{\partial x_\parallel} \,\frac{\partial P}{\partial v_\parallel}\,=\,0,
$$
which completes the first part of the proof.

When the solvability condition is satisfied, then the  equation
(\ref{order-0}), can be solved explicitly and the solution $f_2$ only
depends on $(F,\bE_F)$ and $(P,\bE_P)$ and a function $R\in \ker \mathcal{L}$.
\end{proof}
 
\section{Open problems and conclusion}
\label{sec:5}
In this paper we studied the long time behavior of the solution to the
Vlasov-Poisson system (\ref{V_varepsilon_f}) with a strong external
magnetic field $\bB_{\rm ext}=(0,0,b)$. We provide a formal analysis
based on a Hilbert type expansion of the solution and the formal limit
is solution to a reduced kinetic model (\ref{transport:F}) where the
solution does not depend anymore on the angular perpendicular velocity
$\theta\in (0,2\pi)$. As fas as we know, this reduced model is new and satisfies some fundamental
properties as the correct drift velocities $\bE\times \bB_{\rm ext}$,
gradient B-drift, conservation of energy, entropy, and invariance of
the magnetic moment. Thus, the reduced model  (\ref{transport:F})
seems to be completely relevant  and certainly
deserves more attention. In particular  several important questions remain open.

\subsubsection*{\bf About the generalization to an arbitrary external
  magnetic field.}  Here we focus on the
formal analysis when the magnetic field only applies in the
$z$-direction. We may also consider an arbitrary external magnetic field
  $\bB_{\rm ext}$ in order to get a more elaborated limiting system
  taking into account curvature drift, polarization effects, etc. This
  can be done by following the guideline of the analysis performed in \cite{DHV}.
  
\subsubsection*{\bf About the rigorous justification of the limiting
  system.} 
To justify our asymptotic analysis,  we should consider a smooth
  solution to the Vlasov-Poisson system  (\ref{V_varepsilon_f}), and
 should  assume  that the limiting system also admits  a smooth solution : for any $k\geq 0$,
  $(F,\bE_F) \in {\mathcal C}^{k+3}_c\times  {\mathcal C}^{k+3}$ and
  $(P,\bE_P) \in  {\mathcal C}^{k+2}_c\times  {\mathcal
    C}^{k+2}$, where  ${\mathcal C}^{k}$ is the space of
functions with $k$ continuous derivatives and ${\mathcal C}^{k}_c$ the
sub-space of  ${\mathcal C}^{k}$ with compactly supported functions. 

Then,  we construct  $(F^\varepsilon, \bE^\varepsilon)$  by 
$$
F^\varepsilon \,=\, F  \,+\, \varepsilon\,  f_1 \, +\, \varepsilon^2\, f_2, \quad   \bE^\varepsilon \,=\, \bE_F  + \varepsilon  \,\bE_P,
$$
where $f_1$ and $f_2$ are solutions to (\ref{order-1}) and
(\ref{order-0}), and $f_2$ such that $\Pi f_2 = 0$. Therefore
$(F^\varepsilon, \bE^\varepsilon) \in {\mathcal C}^{k+1}_c\times  {\mathcal C}^{k+1}$  satisfies the Vlasov-Poisson
system    (\ref{V_varepsilon_f}) with a  source term $(R^\varepsilon)_{\varepsilon>0}$
$$
\left\{
\begin{array}{l}
\displaystyle \varepsilon\,\frac{\partial F^\varepsilon}{\partial t}\,+\,\vv \cdot \nabla_\xx F^\varepsilon \,+\,
\left(\bE^\varepsilon\,+\,\frac{b}{\varepsilon}\vv^\perp\right) \cdot
\nabla_\vv F^\varepsilon\,=\, -\varepsilon \,R^\varepsilon,
\\
\,
\\
\ds -\Delta \phi^\varepsilon = \rho^\varepsilon = \int_{\RR^3} F^\varepsilon d\vv, \quad \bE^\varepsilon = -\nabla_\xx
\phi^\varepsilon,
\\
\,
\\
F^\varepsilon(0) \,=\, F(0)+\varepsilon \,f_1(0) +\varepsilon^2 \,f_2(0), 
\end{array}\right.
$$
with 
$$
R^\varepsilon \,=\, \varepsilon\,\left(\frac{\partial f_1}{\partial t}\,+\,\bE_P\cdot\nabla_\vv f_1\right)
\,+\,\varepsilon^2 \,\left(\frac{\partial f_2}{\partial t} + \,+\,\vv
  \cdot \nabla_\xx f_2 \,+\,\bE^\varepsilon\cdot\nabla_\vv f_2\right),
$$ 
such that  for all $k\geq 0$,
$$
\|R^\varepsilon\|_{H^k} \leq C \left[\, \varepsilon \left( \|F\|_{H^{k+2}} +
  \|\bE_F\|_{H^{k+2}} \right) \,+\,  \varepsilon^2 \left( \|F\|_{H^{k+3}} +
  \|\bE_F\|_{H^{k+3}} \right)\,\right].
$$
The second step is to establish a comparison principle on the
Vlasov-Poisson system  (\ref{V_varepsilon_f}) to prove the
convergence
$$
\|F^\varepsilon -  f^\varepsilon\|_{H^k} \leq C\,\|R^\varepsilon\|_{H^k}.
$$

\subsubsection*{\bf About existence, uniqueness and regularity of solutions to the limiting system
  (\ref{transport:F}).} 
The lack of estimates on the Lagrange
  multiplier $(P,\bE_P)$ is the main issue to prove existence of weak
  solutions. Unfortunately, the constraint introduced in the limiting
  system   (\ref{transport:F})  is not standard in kinetic theory and
  fluid mechanics since the constraint is nonlinear in the sense that the
  differential operator ${\mathcal T}_F $ applied to $P$ depends on the solution itself
  {\it via} the potential $\phi_F$
$$
{\mathcal T}_F P := v_\parallel \frac{\partial P}{\partial x_\parallel} 
\,-\,  \frac{\partial \phi_F}{\partial x_\parallel} \frac{\partial P}{\partial v_\parallel}.
$$ 
Therefore, we cannot simply eliminate the
  constraint by introducing an appropriate functional space. New
  estimates have to be established to fix this issue.

\subsubsection*{\bf About the long time behavior of the solution to the limiting system
  (\ref{transport:F}).} The constraint in the parallel direction to the
  magnetic field 
$$
v_\parallel \frac{\partial F}{\partial x_\parallel} 
\,-\,  \frac{\partial \phi_F}{\partial x_\parallel} \frac{\partial F}{\partial v_\parallel} = 0,
$$
is very unusual and may be very strong. Therefore,  it is not guaranteed that the limiting system
  (\ref{transport:F}) can describe accurately plasma turbulence from
  current and spatial gradients and a stability analysis of
  the particular solutions may be investigated. For instance, consider the
   limiting system (\ref{transport:F}) , with $b=1$ and periodic
  boundary conditions in space  with an external background $\rho^\varepsilon_0$ in the
  Poisson equation
$$
-\Delta \phi = \int_{\RR^3}
F d\vv - \rho_0,\quad {\rm with }\quad
\rho_0 = \frac{1}{{\rm m}(\TT^3)}\int_{\TT^3\times\RR^3}
F(0) d\vv\,d\xx.
$$
We choose $F \equiv {\mathcal
  G} (t,\xx_\perp)\, {\mathcal M}(w^2+v_\parallel)$, where ${\mathcal
  G}$ is solution to the guiding center equation
$$
\ds\frac{\partial {\mathcal G}}{\partial t}\,+ \,
\bU_\perp\cdot\nabla_{\xx_\perp} {\mathcal G} = 0,
$$ 
with $\bU_\perp=-\nabla^\perp\phi_F$ and $\mathcal M$ is an arbitrary
smooth and nonnegative function. Thus, we set  $(P,\bE_P)=(0,0)$ and
since  $(F,\bE_F)$ does not depend on $x_\parallel$, the constraint in
(\ref{transport:F}) is automatically satisfied and $(F,\bE_F)$ is
solution to   (\ref{transport:F}). An interesting question is the
stability of such a solution.

\subsubsection*{\bf About the numerical simulation of  (\ref{transport:F}).} 
Finally the numerical approximation of the limiting system
  (\ref{transport:F}) should be investigated to study the relevance of
  such a model. This system has a clear advantage from a
  numerical point of view since the stiffness due to the external
  magnetic field of the Vlasov-Poisson system  (\ref{V_varepsilon_f})
  has been removed and the fast variable $\theta\in (0,2\pi)$ is eliminated
  by averaging. However, the discretization of     (\ref{transport:F})
  is not straightforward due to the constraint in the parallel
  direction to the magnetic field and a specific investigation have to
  be done. One possibility is to follow the strategy  applied in fluid mechanics for
  the two dimensional incompressible Euler system \cite{euler2D}.   

\section*{Acknowledgments.}
Both authors wish to express their gratitude
to Eric Sonnendru\"ucker for fruitful discussions on this topic.   

This work has been supported by the Engineering and Physical Sciences
Research Council (EPSRC) under grant ref. EP/M006883/1, and by the National
Science Foundation (NSF) under grant RNMS11-07444 (KI-Net). 

This work has also been carried out within the framework of the EUROfusion Consortium and has received funding from the Euratom research and training programme 2014-2018 under grant agreement No 633053. The views and opinions expressed herein do not necessarily reflect those of the European Commission.

F. F. is grateful to the Department of Mathematics of Imperial College
of London for its hospitality.  P. D. is on
leave from CNRS, Institut de Math\'ematiques, Toulouse, France. He
acknowledges support from the Royal Society and the Wolfson foundation
through a Royal Society Wolfson Research Merit Award.


\begin{flushleft} 
\signPD 
\end{flushleft}
\begin{flushright} 
\signFF 
\end{flushright}

\end{document}